\documentclass{article}

\usepackage{amsmath,amssymb,amsthm,mathrsfs}

\newcommand{\cH}{{\mathcal H}}

\newcommand{\cK}{{\mathcal K}}

\newcommand{\cU}{{\mathcal U}}
\newcommand{\cX}{{\mathcal X}}
\newcommand{\cY}{{\mathcal Y}}

\newcommand{\wtilE}{\widetilde{E}}\newcommand{\wtilF}{\widetilde{F}}

\newcommand{\wtilT}{\widetilde{T}}
\newcommand{\wtilU}{\widetilde{U}}\newcommand{\wtilV}{\widetilde{V}}
\newcommand{\wtilX}{\widetilde{X}}

\newcommand{\whatU}{\widehat{U}}\newcommand{\whatV}{\widehat{V}}

\newcommand{\de}{\delta}
\newcommand{\vep}{\varepsilon}

\newcommand{\im}{\textup{Im\,}}

\newcommand{\mat}[2]{\ensuremath{\left[\begin{array}{#1}#2\end{array} \right]}}

\newcommand{\sbm}[1]{\left[\begin{smallmatrix} #1\end{smallmatrix}\right]}
\newcommand{\ov}[1]{{\overline{#1}}}

\newcommand{\wtil}[1]{{\widetilde{#1}}}

\newcommand{\half}{\frac{1}{2}}
\newcommand{\ands}{\quad\mbox{and}\quad}
\newcommand{\ons}{\mbox{ on }}

\theoremstyle{plain}
\newtheorem{theorem}{Theorem}[section]
\newtheorem{corollary}[theorem]{Corollary}
\newtheorem{lemma}[theorem]{Lemma}
\newtheorem{proposition}[theorem]{Proposition}
\theoremstyle{definition}
\newtheorem{definition}[theorem]{Definition}
\newtheorem{example}[theorem]{Example}

\newcommand{\ts}{{\times}}

\newcommand{\dpl}{\dotplus}

\begin{document}

\title{Equivalence after extension and matricial coupling coincide with
Schur coupling, on separable Hilbert spaces}

\author{Sanne ter Horst\footnotemark[1], and Andr\'e C.M. Ran\footnotemark[2]}
\renewcommand{\thefootnote}{\fnsymbol{footnote}}
\footnotetext[1]{School~of~Computer,~Statistical~and~Mathematical~Sciences,
North-West~University,
Research unit for BMI,
Private~Bag~X6001,
Potchefstroom~2520,
South Africa.
E-mail: \texttt{sanne.terhorst@nwu.ac.za}
}
\footnotetext[2]{Department of Mathematics, FEW, VU university Amsterdam, De Boelelaan
    1081a, 1081 HV Amsterdam, The Netherlands
    and Unit for BMI, North-West~University,
Potchefstroom,
South Africa. E-mail:
    \texttt{a.c.m.ran@vu.nl}}

\date{}


\maketitle
\begin{abstract}
It is known that two Banach space operators that are Schur coupled are also equivalent after extension, or equivalently, matricially coupled. The converse implication, that operators which are equivalent after extension or matricially coupled are also Schur coupled, was only known for Fredholm Hilbert space operators and Fredholm Banach space operators with index 0. We prove that this implication also holds for Hilbert space operators with closed range, generalizing the result for Fredholm operators, and Banach space operators that can be approximated in operator norm by invertible operators. The combination of these two results enables us to prove that the implication holds for all operators on separable Hilbert spaces.
\end{abstract}



\section{Introduction}
\setcounter{equation}{0}

This paper deals with a question raised in \cite{BT92b}, see also \cite{BT94,BGKR05}, concerning certain operator relations for bounded linear Banach space operators. In the sequel the term operator will be short for bounded linear operator and an invertible operator implies the inverse is a (bounded) operator as well. Moreover, by the term subspace we shall mean a closed linear manifold.

The operator relations in question are listed in the following definition.

\begin{definition}\label{D:eqrels}
Let $U$ on $\cX$ and $V$ on $\cY$ be Banach space operators.
\begin{itemize}
\item[(SC)]
The operators $U$ and $V$ are called {\em Schur coupled} in case there exists an operator matrix
\[
S=\mat{cc}{A&B\\C&D}\ons \cX\dpl\cY
\]
with $A$ and $D$ invertible,
\[
U=W_{2,2}(S):=A-BD^{-1}C \ands V=W_{1,1}(S):=D-CA^{-1}B.
\]
Here and in the sequel we use the notation $W_{i,j}(T)$, $i,j=1,2$, to indicate the {\em Schur complement} of a $2\ts 2$ operator matrix $T$ with respect to the $(i,j)$-th entry, provided the $(i,j)$-th entry is invertible.

\item[(EAE)]
The operators $U$ and $V$ are called {\em equivalent after extension} if there exist Banach spaces $\cX_0$ and $\cY_0$ and invertible operators $E$ mapping $\cY\dpl\cY_0$ onto $\cX\dpl\cX_0$ and $F$ mapping $\cX\dpl\cX_0$ onto $\cY\dpl\cY_0$ such that
\[
\mat{cc}{U&0\\0&I_{\cX_0}}=E \mat{cc}{V&0\\0&I_{\cY_0}} F.
\]

\item[(MC)] The operators $U$ and $V$ are called
{\em matricially coupled} if there exist an invertible operator
\begin{equation*}
\whatU\!=\!\mat{cc}{U_{11}&U_{12}\\U_{21}&U_{22}}\!\!\ons \cX\dpl\cY,
\mbox{ with inverse } \whatU^{-1}\!=\!\mat{cc}{V_{11}&V_{12}\\V_{21}&V_{22}}\!\!\ons
\cX\dpl\cY
\end{equation*}
such that $U=U_{11}$ and $V=V_{22}$.
\end{itemize}
\end{definition}

These three operator relations can be defined for operators acting between different Banach spaces as well, with this limitation that for the notion of Schur coupling to make sense the Banach spaces need to be isomorphic, since $A$ and $D$ are to be invertible operators. Letting $U$ and $V$ be operators that map a Banach space into itself may seem restrictive, but it is natural from the viewpoint of the question posed in \cite{BT92b}, and it makes the notation more managable. In the interest of readability we have decided to restrict ourselves to this case (as is also done in \cite{BGKR05}).

The question raised in \cite{BT92b} is whether these three operator relations coincide. Most implications in the proof of the equivalence of these three notions are known. It was proved in \cite{BGK84} that matricial coupling implies equivalence after extension, see also \cite[Section III.4]{GGK90}. The converse implication was proved in \cite{BT92a}. Hence the operator relations matricial coupling and equivalence after extension coincide. Consequently, matricial coupling is an equivalence relation, which is obviously the case for equivalence after extension, but less obviously so for matricial coupling. It was further proved, in \cite{BT92b}, that Schur coupling implies the other two operator relations, and that the converse holds under some additional constraints, leading to the notions of strong matricial coupling and strong equivalence after extension, which will be defined in Definition \ref{D:SMCnSEAE} below. All the above implications can be proved by rather constructive arguments. The details of these constructions will be given in Section \ref{S:review}. The remaining implication, equivalence after extension (= matricial coupling) implies Schur coupling, is still open in the general case. However, the implication was proved for matrices in \cite{BT94} and in \cite{BGKR05} for Fredholm operators acting between Hilbert spaces as well as Banach space Fredholm operators with index 0.

The following theorem is the main result of the present paper.

\begin{theorem}\label{T:main}
The operator relations Schur coupling, matricial coupling and equivalence after extension coincide for all Hilbert space operators acting between separable Hilbert spaces.
 \end{theorem}

The proof of Theorem \ref{T:main} will be given in Section \ref{S:MT}. It is based on two results which prove the remaining implication ((EAE) $\Rightarrow$ (SC)) for two specific classes of operators. These results are given in Theorems \ref{T:InvApproxCase} and \ref{T:MPcase} below. A theorem of Feldman and Kadison \cite{FK54} implies that all Hilbert space operators acting between separable Hilbert spaces are in either one of these two classes, hence Theorem \ref{T:main} follows. On non-separable Hilbert spaces there exist operators that are not contained in either of the two classes, which will be illustrated with an example in Section \ref{S:MT}.

The next theorem deals with the first class of operators for which the remaining implication holds. A proof is given in Section \ref{S:GenInvCase}.

\begin{theorem}\label{T:MPcase}
Assume $U$ on $\cX$ and $V$ on $\cY$ are Hilbert space operators with closed range. Then the following are equivalent:
\begin{itemize}
\item[(i)] $U$ and $V$ are Schur coupled;

\item[(ii)] $U$ and $V$ are equivalent after extension;

\item[(iii)] There exist invertible operators
\[
K:\ker U \to \ker V,\quad
K_*:\ker U^*\to\ker V^*.
\]
\end{itemize}
\end{theorem}

Note that a Hilbert space operator has closed range if and only if it has a Moore-Penrose generalized inverse. In Section \ref{S:GenInvCase} we also prove a result for Banach space operators, see Theorem \ref{T:GIcase}. Here the closed range assumption is replaced by the assumption that the operators are generalized invertible, in the sense of \cite{N87}. However, it will also be assumed that there exists an isomorphism between $\ker U$ and $\cX/\im U$ and likewise between  $\ker V$ and $\cY/\im V$. Theorems \ref{T:MPcase} and \ref{T:GIcase} can be seen as direct generalizations of the results in \cite{BT92b}, see also \cite{BGKR05}, for Fredholm operators. In this case the additional constraint replaces the assumption in \cite{BT92b} that the Banach space Fredholm operators have index 0.

The second class of operators considered in this paper are those which can be approximated by invertible operators in the operator norm. Since the operator norm is the only norm that plays a role in the present paper, we will simply speak of `the norm' when indicating the operator norm and we will denote the norm of an operator $T$ by $\|T\|$. The following theorem will be proved in Section \ref{S:InvApprox}.

\begin{theorem}\label{T:InvApproxCase}
Let $U$ on $\cX$ and $V$ on $\cY$ be two Banach space operators that are equivalent after extension. Assume $U$ or $V$ can be approximated in norm by invertible operators. Then $U$ and $V$ are Schur coupled.
\end{theorem}

Note that while all operators on finite dimensional Hilbert spaces can be approximated with invertible operators, this not the case on any infinite Hilbert space. Indeed, any shift operator, or, more generally, any isometry which is not unitary, cannot be approximated by invertible operators. In Section \ref{S:MT} we recall some necessary and sufficient conditions from the literature under which an operator can be approximated by invertible operators.

\section{Review of the known implications}\label{S:review}
\setcounter{equation}{0}

In this section we discuss the known equivalence implications between the operator relations of Definition \ref{D:eqrels}. All these implications follow by explicit constructions. The proofs that the constructed operators have the required properties are straightforward, and verification is left to the reader. Although these implications are known, their proofs are not always easily accessible. Moreover, some of the constructions provided in this section are slightly different from the ones appearing in the literature, and they are essential in our proofs.

We start with the equivalence of matricial coupling and equivalence after extension.

\begin{theorem}\label{T:MCeqsEAE}
Let $U$ on $\cX$ and $V$ on $\cY$ be Banach space operators. Assume $U$ and $V$ are matricially coupled via operators $\whatU$ and $\whatV$ as in (MC). Set
\begin{equation}\label{EFspecialform1}
E=\mat{cc}{U_{12}&U\\U_{22}&U_{21}}\ands
F=\mat{cc}{-U_{21}&I_\cY\\V_{11}U&V_{12}}.
\end{equation}
Then $E$ and $F$ are invertible operators, with inverses given by
\[
E^{-1}=\mat{cc}{V_{21}&V\\V_{11}&V_{12}},\quad
F^{-1}=\mat{cc}{-V_{12}&I_\cX\\U_{22}V&U_{21}},
\]
and we have
\[
\mat{cc}{U&0\\0&I_{\cY}}=E\mat{cc}{V&0\\0&I_{\cX}}F.
\]
In particular, $U$ and $V$ are equivalent after extension.

Conversely, assume $U$ and $V$ are equivalent after extension, via operators $E$ and $F$ as in (EAE). Decompose $E$, $F$ and their inverses:
\begin{equation}\label{EFEFinvdec}
\begin{aligned}
E=\mat{cc}{E_{11}&E_{12}\\E_{21}& E_{22}},\  F^{-1}=\mat{cc}{F_{11}^{(-1)}&F_{12}^{(-1)}\\F_{21}^{(-1)}&F_{22}^{(-1)}} :\cY\dpl\cY_0\to\cX\dpl\cX_0,\\
F=\mat{cc}{F_{11}&F_{12}\\F_{21}& F_{22}},\  E^{-1}=\mat{cc}{E_{11}^{(-1)}&E_{12}^{(-1)}\\E_{21}^{(-1)}&E_{22}^{(-1)}} :\cX\dpl\cX_0\to\cY\dpl\cY_0.
\end{aligned}
\end{equation}
Set
\begin{align*}
\whatU=\mat{cc}{U& -E_{11}\\F_{11}&F_{12}E_{21}}\ands
\whatV=\mat{cc}{F_{12}^{(-1)}E_{21}^{(-1)}&F_{11}^{(-1)}\\-E_{11}^{(-1)}&V}.
\end{align*}
Then $\whatU$ and $\whatV$ are invertible and $\whatU^{-1}=\whatV$. In particular, $U$ and $V$ are matricially coupled.
\end{theorem}

The next theorem provides the constructions that prove that Schur coupling implies the other two operator relations.

\begin{theorem}\label{T:SCtoEAEnMC}
Assume $U$ on $\cX$ and $V$ on $\cY$ are Schur coupled via the block operator matrix $S=\sbm{A&B\\C&D}$ with $A$ and $D$ invertible. Set
\[
\begin{array}{c}
E=\mat{cc}{-BD^{-1}&U\\ D^{-1}&D^{-1}C},\quad
F=\mat{cc}{-D^{-1}C&I\\A^{-1}U&A^{-1}B},\\[.4cm]
\whatU=\mat{cc}{U&-BD^{-1}\\ D^{-1}C&D^{-1}}.
\end{array}
\]
Then $E$, $F$ and $\whatU$ are invertible with inverses
\[
\begin{array}{c}
E^{-1}=\mat{cc}{-CA^{-1}&V\\A^{-1}&A^{-1}B},
\quad F^{-1}=\mat{cc}{-A^{-1}B&I_\cX\\D^{-1}V&D^{-1}C},\\[.4cm]
\whatU^{-1}=\mat{cc}{A^{-1}&A^{-1}B \\ -CA^{-1} & V}.
\end{array}
\]
and
\[
\mat{cc}{U&0\\0&I_{\cY}}=E \mat{cc}{V&0\\0&I_{\cX}} F
\]
In particular, $U$ and $V$ are equivalent after extension and matricially coupled.
\end{theorem}

Following \cite{BT92b} we introduce the notions of strong matricial coupling and strong equivalence after extension.

\begin{definition}\label{D:SMCnSEAE}
Let $U$ on $\cX$ and $V$ on $\cY$ be Banach space operators. Then $U$ and $V$ are called {\em strongly matricially coupled} if $U$ and $V$ are matricially coupled as in (MC) and in addition the operators $U_{22}$ and $V_{11}$ are invertible. Moreover, $U$ and $V$ are called {\em strongly equivalent after extension} if $U$ and $V$ are equivalent after extension as in (EAE) and in addition the operators $E_{21}$ and $F_{12}$ in the operator matrix decompositions of $E$ and $F$ in \eqref{EFEFinvdec} are invertible.
\end{definition}

From the constructions in Theorem \ref{T:SCtoEAEnMC} it is clear that Schur coupling implies the stronger notions of matricial coupling and equivalence after extension. It was proved in \cite[Theorem 2]{BT92b} that these three notions in fact coincide. Details are given in the following result.

\begin{theorem}\label{T:SMCnSEAEeqsSC}
The operator relations (i) Schur coupling, (ii) strong matricial coupling, and (iii) strong equivalence after extension coincide. To be precise, let $U$ on $\cX$ and $V$ on $\cY$ be Banach space operators.

Assume $U$ and $V$ are strongly matricially coupled. Set $S=\sbm{A&B\\C&D}$ with
\[
A=W_2(\whatU)=U-U_{12}U_{22}^{-1}U_12,\quad B=-U_{12}U_{22}^{-1},\quad
C=U_{22}^{-1} U_{21},\quad D=U_{22}^{-1}.
\]
Then $A$ and $D$ are invertible, $U=W_{2,2}(S)$ and $V=W_{1,1}(S)$, i.e.,
$U$ and $V$ are Schur coupled.

Assume $U$ and $V$ are strongly equivalent after extension. Then the operators $E_{21}^{(-1)}$ and $F_{12}^{(-1)}$ in \eqref{EFEFinvdec} are invertible with inverses
\[
(E_{21}^{(-1)})^{-1}=E_{12}-E_{11}E_{21}^{-1}E_{22}\ands
(F_{12}^{(-1)})^{-1}=F_{21}-F_{22}F_{12}^{-1}F_{11}.
\]
Set
\[
S=\mat{cc}{A&B\\C&D}=\mat{cc}{(E_{21}^{(-1)})^{-1}(F_{12}^{(-1)})^{-1} & -E_{11}E_{21}^{-1}F_{12}^{-1} \\ E_{21}^{-1}E_{22}(F_{12}^{(-1)})^{-1} & E_{21}^{-1}F_{12}^{-1}}.
\]
Then $A$ and $D$ are invertible, $U=W_{2,2}(S)$ and $V=W_{1,1}(S)$, i.e.,
$U$ and $V$ are Schur coupled.

The implications (i) $\Rightarrow$ (ii) and (i) $\Rightarrow$ (iii)
follow from the constructions in Theorem \ref{T:SCtoEAEnMC}.
The implications (ii) $\Leftrightarrow$ (iii) follow from the
constructions in Theorem \ref{T:MCeqsEAE}.
\end{theorem}

A consequence of Theorem \ref{T:SMCnSEAEeqsSC} is that in order to prove that equivalence after extension, or matricial coupling, implies Schur coupling, it suffices to prove that equivalence after extension implies strong equivalence after extension, or that matricial coupling implies strong matricial coupling.

\section{Proof of Theorem \ref{T:MPcase}}\label{S:GenInvCase}
\setcounter{equation}{0}

In this section we consider the case that the operators $U$ and $V$ admit a generalized inverse. We start with the case of Banach space operators, and will assume that $U$ and $V$ are generalized invertible operators, i.e., we assume there exist operators $U^+$ on $\cX$ and $V^+$ on $\cY$ such that
\[
U=UU^+U,\quad U^+=U^+UU^+\ands
V=VV^+V,\quad V^+=V^+VV^+.
\]
We now recall another characterization of generalized invertibility. The operators $U$ and $V$ are generalized invertible if and only if the spaces $\ker U$ and  $\im U$ are complemented in $\cX$ and $\ker V$ and $\im V$ are complemented in $\cY$, cf., \cite[Theorem XI.6.1]{GGK90}, that is, there exist subspaces $\cX_1,\cX_2\subset\cX$ and $\cY_1,\cY_2\subset\cY$ such that
\begin{equation}\label{SpacesDec}
\cX=\cX_1\dpl\ker U=\im U\dpl\cX_2\ands
\cY=\cY_1\dpl\ker V=\im V\dpl\cY_2.
\end{equation}
Moreover, the subspaces $\cX_1$, $\cX_2$, $\cY_1$ and $\cY_1$ can be chosen in such a way that with respect to
the decompositions in \eqref{SpacesDec} the operators $U$ and $V$ are of the form
\begin{equation}\label{UVdec}
\begin{aligned}
&U=\mat{cc}{U_1&0\\0&0}:\cX_1\dpl\ker U\to\im U\dpl\cX_2,
\quad\mbox{with $U_1$ invertible;}\\
&V=\mat{cc}{V_1&0\\0&0}:\cY_1\dpl\ker V\to\im V\dpl\cY_2,
\quad\mbox{with $V_1$ invertible.}
\end{aligned}
\end{equation}
In that case $U^+=\sbm{U_1^{-1}&0\\0&0}$ and $V^+=\sbm{V_1^{-1}&0\\0&0}$
are generalized inverses of $U$ and $V$, respectively.

Generalized invertibility is a natural concept to consider in the
context of equivalence after extension, in the following sense, see
\cite{BT92a}.

\begin{lemma}
Assume $U$ and $V$ are equivalent after extension. Then $U$ is generalized invertible if and only if $V$ is generalized invertible.
\end{lemma}

\begin{proof}[\bf Proof.]
Since equivalence after extension is equivalent to matricial coupling, we have that
$U$ and $V$ are matricially coupled. The lemma then follows from Theorem 2.1 in \cite{BGK84} (compare also \cite{GGK90}, Section III.4).
\end{proof}

The next proposition explains the structure of the operators $E$ and $F$ appearing in the definition of equivalence after extension with respect to the block operator decompositions in \eqref{UVdec}.

\begin{proposition}\label{P:EAEstructure}
Let $U$ on $\cX$ and $V$ on $\cY$ be generalized invertible Banach space operators. Assume $U$ and $V$ are equivalent after extension as in (EAE). Then with respect to a decomposition as in \eqref{UVdec} the invertible operators $E$ and $F$ in (EAE) admit operator matrix representations of the form
\begin{equation}\label{EFdec}
\begin{aligned}
&E=\mat{cc|c}{
E_{11}^a&E_{11}^b&E_{12}^a\\0&E_{11}^c&0\\ \hline E_{21}^a&E_{21}^b&E_{22}}
:(\im V\dpl \cY_2)\dpl \cY_0\to (\im U\dpl \cX_2)\dpl\cX_0,\\
&F=\mat{cc|c}{F_{11}^a&0&F_{12}^a\\F_{11}^b&F_{11}^c&F_{12}^b\\ \hline F_{21}^a&0&F_{22}}
:(\cX_1\dpl\ker U) \dpl \cX_0\to(\cY_1\dpl \ker V)\dpl \cY_0,
\end{aligned}
\end{equation}
with $E_{11}^c$, $F_{11}^c$, $\sbm{E_{11}^a&E_{12}^a\\ E_{21}^a& E_{22}}$ and
$\sbm{F_{11}^a&F_{12}^a\\ F_{21}^a& F_{22}}$ invertible.
Moreover, we may without
loss of generality assume the operators
\begin{equation}\label{rep01}
E_{11}^b,\quad E_{12}^b,\quad F_{11}^b,\quad F_{12}^b
\end{equation}
to be zero, that is, if $E$ and $F$ have the above form
and the operators in \eqref{rep01} are replaced by zero operators, then the
resulting operators $E$ and $F$ are still invertible and (EAE) remains
to hold with the modified operators $E$ and $F$.
\end{proposition}

\begin{proof}[\bf Proof.]
The fact that $F\sbm{V&0\\0&I_{\cY_0}}E=\sbm{U&0\\0&I_{\cX_0}}$
holds with $E$ and $F$ invertible shows that $F$ maps $\ker U$
boundedly invertible onto $\ker V$. This shows that the $(1,2)$-
and $(3,2)$-entries in the block operator decomposition of $F$
in \eqref{EFdec} are zero operators and that $F_{11}^c$ is
invertible. Similarly one finds the zero entries in the block
operator decomposition of $E$ in \eqref{EFdec} and the fact that
$E_{11}^c$ is invertible.  Due to the zero entries in \eqref{EFdec}
the Schur complements in $E$ with respect to $E_{11}^c$ and in
$F$ with respect to $F_{11}^c$ equal
$E^a:=\sbm{E_{11}^a&E_{12}^a\\ E_{21}^a& E_{22}}$ and
$F^a:=\sbm{F_{11}^a&F_{12}^a\\ F_{21}^a& F_{22}}$, respectively.
Since $E$ and $F$ are invertible, so are the Schur complements
$E^a$ and $F^a$.

Since $E_{11}^c$, $F_{11}^c$, $E^a$ and $F^a$ are all invertible,
replacing the operators of \eqref{rep01} in $E$ and $F$ by zero
operators, $E$ and $F$ clearly remain invertible. Moreover,
using the decompositions of $U$ and $V$ in \eqref{UVdec} we see that
$F\sbm{V&0\\0&I_{\cY_0}}E=\sbm{U&0\\0&I_{\cX_0}}$ reduces to
\[
\mat{cc}{U^a&0\\0&I_{\cX_0}}=E^a\mat{cc}{V^a&0\\0&I_{\cY_0}}F^a.
\]
Hence the operators in \eqref{rep01} have no part in the validity
of (EAE). Thus (EAE) remains to hold when the operators in \eqref{rep01}
are replaced by zero operators.
\end{proof}

Next we present the extension of Theorem 3 in \cite{BT92b} to the case of
generalized invertible Banach space operators.

\begin{theorem}\label{T:GIcase}
Let $U$ on $\cX$ and $V$ on $\cY$ be generalized invertible Banach space operators. Assume there exists isomorphisms between $\ker U$ and $\cX/\im U$ and between $\ker V$ and $\cY/\im V$. Then with respect to any decomposition as in \eqref{UVdec} the following are equivalent:
\begin{itemize}
\item[(i)] $U$ and $V$ are Schur coupled;

\item[(ii)] $U$ and $V$ are equivalent after extension;

\item[(iii)] There exist an isomorphism $K:\ker U \to \ker V$.
\end{itemize}
\end{theorem}

\begin{proof}[\bf Proof.]
By Theorem \ref{T:SCtoEAEnMC} we have (i) $\Rightarrow$ (ii).
Assuming (ii), we can take $F_{11}^c$ in Proposition \ref{P:EAEstructure} as an isomorphism from $\ker U$ to $\ker V$. Hence (ii) $\Rightarrow$ (iii).

It remains to prove (iii) $\Rightarrow$ (i), for which it suffices to prove that (iii) implies strong matricial coupling, by Theorem \ref{T:SMCnSEAEeqsSC}. We shall follow the argumentation of the proof of Theorem 3 in \cite{BT92b}. Since $U$ and $V$ are generalized invertible, they can be written as in \eqref{UVdec}. Now let $J_U : \ker U \to \cX_2$ and $J_V: \ker V\to \cY_2$ be isomorphisms. Such isomorphisms exist, because of our assumption that there exist isomorphisms between $\ker U$ and $\cX/\im U$ and $\ker V$ and $\cY/\im V$, respectively, and the fact that any complement of $\im U$ is isomorphic to $\cX/\im U$ and, likewise, any complement of $\im V$ is isomorphic to $\cX/\im V$. It follows that $K_2=J_V K J_U^{-1}$ is an isomorphism from $\cX_2$ onto $\cY_2$. Therefore,
\[
\mat{cc|cc}{U_1&0&0&0\\0&0&0&K_2^{-1}\\\hline 0&0&V_1^{-1}&0\\0&-K&0&J_V^{-1}}=
\mat{cc|cc}{U_1^{-1}&0&0&0\\0&K^{-1}J_V^{-1}K_2&0&-K^{-1}\\\hline 0&0&V_1&0\\0&K_2&0&0}^{-1}.
\]
Since $\sbm{V_1^{-1}&0\\ 0&J_V^{-1}}$ and $\sbm{U_1^{-1}&0\\ 0&K^{-1}J_V^{-1}K_2}$ are both invertible, it follows that $U$ and $V$ are strongly matricially coupled by taking for $\whatU$ the left hand side in the above identity.
\end{proof}

Note that the implication (i) $\Rightarrow$ (ii) $\Rightarrow$ (iii) were proved without using the additional assumptions that $\ker U$ and $\cX/\im U$ are isomorphic as well as   $\ker V$ and $\cY/\im V$.

For Hilbert space operators, the notions of generalized
invertibility coincides with Moore-Penrose invertibility, or equivalently,
with the condition that the operator in question has closed range. Moreover,
if $U$ and $V$ are generalized invertible Hilbert space operators, then we
can take
\begin{equation}\label{HilbertChoice}
\begin{aligned}
\cX_1& =\cX\ominus\ker U=\im U^*,\quad
\cX_2=\cX\ominus\im U=\ker U^*,\\
\cY_1& =\cY\ominus\ker V=\im V^*,\quad
\cY_2=\cY\ominus\im V=\ker V^*,
\end{aligned}
\end{equation}
in the representations \eqref{UVdec} of $U$ and $V$.

\begin{proof}[\bf Proof of Theorem \ref{T:MPcase}.]
(i) $\Rightarrow$ (ii) follows from Theorem \ref{T:SCtoEAEnMC} and
(ii) $\Rightarrow$ (iii) follows the invertibility of the operators
$E_{11}^c$ and $F_{11}^c$ in Proposition \ref{P:EAEstructure}, along with the
choice of \eqref{HilbertChoice} for the complements in the representations \eqref{UVdec} of $U$ and $V$.

It remains to prove (iii) $\Rightarrow$ (i). Hence assume that
invertible operators $K$ and $K_*$ as in (iii) exist. Define $\cX_1$, $\cX_2$,
$\cY_1$ and $\cY_2$ as in \eqref{HilbertChoice}.
Without loss of generality we may assume the Schauder dimension of $\im U$
does not exceed the Schauder dimension of $\im V$. This implies there
exists an isometry $L:\im U\to \im V$. Let $\cX_0$ be a closed complement
of $\im L$ in $\im V$ such that $L=\sbm{L_1\\0}:\im U\to\sbm{\im L\\\cX_0}$ with
$L_1$ invertible. (One can take for instance for $\cX_0$ the orthogonal complement of $\im L$, in $\im V$. It is here that we use the fact that we are working with Hilbert space operators.) Then
\begin{align*}
&E=\mat{ccc}{L_1^{-1}&0&0\\0&0&K_*^{-1}\\0&I_{\cX_0}&0}
:\im L\dpl\cX_0\dpl\cY_2 \to \im U\dpl \cX_2\dpl\cX_0,\\
&F=\mat{cc}{V_1^{-1}&0\\0& I_{\ker V}}
\mat{cc|c}{L_1 U_1&0&0\\ 0&0&I_{\cX_0}\\ \hline 0&K&0}
:(\cX_1\dpl \ker U)\dpl\cX_0\to\cY_1\dpl \ker V
\end{align*}
are invertible operators, $E$ mapping $\cY=\im V \dpl \cY_2=(\im L\dpl \cX_0)\dpl \cY_2$ onto $\cX\dpl \cX_0=(\im U\dpl \cX_2)\dpl \cX_0$ and $F$ mapping $\cX\dpl\cX_0=(\cX_1\dpl \ker U)\dpl\cX_0$ onto $\cY=\cY_1\dpl\ker V$ via the intermediate space $\im V\dpl\ker V=\im L\dpl \cX_0\dpl \ker V$. Moreover, one easily verifies that
\[
E^{-1}\mat{cc}{U&0\\0& I_{\cX_0}}=V F.
\]
Hence $U$ and $V$ are equivalent after one-sided extension. However, Proposition 5 in
\cite{BGKR05} shows that equivalence after one-sided extension implies Schur coupling.
Indeed, following the construction from \cite[Proposition 5]{BGKR05}, let $P=\mat{cc}{I_\cX&0}$ be the canonical projection from $\cX\dpl \cX_0$ onto $\cX$ and $J=\sbm{I_{\cX}\\0}$ the canonical embedding of $\cX$ into $\cX\dpl \cX_0$. Set
\[
S=\mat{cc}{A&B\\C&D}=\mat{cc}{I_\cX & PF^{-1}\\ E^{-1}J(I_\cX-U)& E^{-1}F^{-1}} \ons \cX\dpl \cY.
\]
Clearly $A$ and $D$ are invertible, and it is a straightforward computation to verify that $U=W_{2,2}(S)$ and $V=W_{1,1}(S)$. Hence $U$ and $V$ are Schur coupled.
\end{proof}

As observed in the proof of Theorem \ref{T:MPcase}, the complication in the case that $U$ and $V$ are Banach space operators is that there may not exist a closed complement $\cX_0$ of the range of the isometry $L$. In fact, a careful inspection of the proof shows that this is the only complication that occurs, provided we let the invertible operator $K_*$ act from $\cX_2$ onto $\cY_2$. Hence we can conclude that Theorem \ref{T:MPcase} holds with $U$ and $V$ Banach space operators, with $\ker U^*$ and $\ker V^*$ replaced by the spaces $\cX_2$ onto $\cY_2$, respectively, from the representations of $U$ and $V$ in \eqref{UVdec}, under the additional assumption that there exists a generalized invertible isometry from $\im U$ into $\im V$ or a generalized invertible isometry from $\im V$ into $\im U$.

\section{Proof of Theorem \ref{T:InvApproxCase}}\label{S:InvApprox}
\setcounter{equation}{0}


As a direct consequence of the constructions in Theorem \ref{T:MCeqsEAE}
we find that the extension spaces in the equivalence after extension
relation can be chosen specifically, as well as some of the entries
of $E$ and $F$.

\begin{lemma}\label{L:toFS}
Assume the Banach space operators $U$ on $\cX$ and $V$ on $\cY$ are equivalent after extension via operators $E$ and $F$ as in (EAE). Decompose $E$ and $F$ and their inverses $E^{-1}$ and $F^{-1}$ as in \eqref{EFEFinvdec}. Then $U$ and $V$ are also equivalent after extension via
the operators
\[
\wtilE=\mat{cc}{-E_{11}&U\\F_{12}E_{21}&F_{11}}\ands
\wtilF=\mat{cc}{-F_{11}&I_\cY\\F_{12}^{(-1)}E_{21}^{(-1)}U & F_{11}^{(-1)}}.
\]
with inverses
\[
\wtilE^{-1}=\mat{cc}{-E_{11}^{(-1)}&V\\F_{12}^{(-1)}E_{21}^{(-1)}&F_{11}^{(-1)}}\ands
\wtilF^{-1}=\mat{cc}{-F_{11}^{(-1)}&I_\cX\\F_{12}E_{21}V & F_{11}}.
\]
\end{lemma}

\begin{proof}[\bf Proof.]
The formulas for $\wtilE$ and $\wtilF$ and their inverses follow by applying the two constructions from Theorem  \ref{T:MCeqsEAE} in reversed order as follows. Starting with $E$ and $F$ as in \eqref{EFEFinvdec}, define $\wtilU$ and $\wtilV$ as in the second part of Theorem \ref{T:MCeqsEAE} (the part which proves the implication (EAE) $\Rightarrow$ (MC)), and consequently apply the formulas from the first part of Theorem \ref{T:MCeqsEAE} (from the converse direction (MC) $\Rightarrow$ (EAE)). This results in the operators $\wtilE$ and $\wtilF$, and their inverses, as follows by a straightforward computation.
\end{proof}

\begin{corollary}\label{C:EAEspecialform}
Assume the Banach space operators  $U$ on $\cX$ and $V$ on $\cY$ are equivalent after extension. Then $U$ and $V$ are equivalent after extension via invertible operators $E$ and $F$ as in (EAE) with
\begin{itemize}
\item[(i)] $\cX_0=\cY$ and $\cY_0=\cX$;

\item[(ii)] $F=\mat{cc}{* & I_{\cY}\\ * & *}$ and
$F^{-1}=\mat{cc}{* & I_{\cX}\\ * & *}$.
\end{itemize}
Here $F$ and $F^{-1}$ are decomposed accordingly to \eqref{EFEFinvdec} and the $*$-s indicate unspecified entries in the operator decompositions. Moreover, in this case the right upper corners in the operator matrix decompositions of $E$ and $E^{-1}$ in \eqref{EFEFinvdec} equal $U$ and $V$, respectively.
\end{corollary}

\begin{proof}[\bf Proof.]
The observation that we can arrange equivalence after extension with (i) and (ii) follows directly from Lemma \ref{L:toFS}. This lemma also shows that $E$ and $E^{-1}$ have $U$ and $V$ in their respective right upper corners. However, we claim that this holds as a consequence of (i) and (ii). Indeed, this follows directly from inspecting the $(1,2)$-entries on both sides of the identities $\sbm{U&0\\0&I}F^{-1}=E\sbm{V&0\\0&I}$ and $E^{-1}\sbm{U&0\\0&I}=\sbm{V&0\\0&I}F$.
\end{proof}

As observed in Section \ref{S:review}, the question whether equivalence after extension implies Schur coupling is answered affirmatively if we can show that equivalence after extension implies strong equivalence after extension, i.e., with $E$ and $F$ such that the operators $E_{21}$ and $F_{12}$ in the operator matrix decompositions \eqref{EFEFinvdec} are invertible. Hence, if $U$ and $V$ are equivalent after extension
via given operators $E$ and $F$ as in (EAE), one can try to modify these operators in such a way that they still prove the equivalence after extension of $U$ and $V$ but with the appropriate operators in their decompositions invertible. However, it is not clear what modifications achieve this. Note that we can always take $E$ and $F$ as in Corollary \ref{C:EAEspecialform}, in particular, it can always be arranged that $F_{12}$ is invertible. The next lemma gives an operation on the operators $E$ and $F$ in (EAE) that preserves the equivalence relation, i.e., equivalence after extension also follows with the operators resulting from the operation, and also preserves the special form of Corollary \ref{C:EAEspecialform}.

\begin{lemma}\label{L:EAEtrans}
Assume $U$ on $\cX$ and $V$ on $\cY$ are equivalent after extension via operators $E$ and $F$ as in (EAE). Let $X:\cX\to\cX_0$ and $Y:\cY\to\cY_0$ and set
\[
\wtilE=\mat{cc}{I_\cX&0\\X&I_{\cX_0}}E\mat{cc}{I_\cY&0\\Y&I_{\cY_0}},\quad
\wtilF=\mat{cc}{I_\cY&0\\-YV&I_{\cY_0}}F\mat{cc}{I_{\cX}&0\\-XU&I_{\cX_0}}.
\]
Then $\wtilE$ and $\wtilF$ are invertible, and
$\sbm{U&0\\0&I_{\cX_0}}=\wtilE \sbm{V&0\\0&I_{\cY_0}}\wtilF$.
Moreover, if $E$ and $F$ are of the special form of Corollary
\ref{C:EAEspecialform}, then so are $\wtilE$ and $\wtilF$.
\end{lemma}

\begin{proof}[\bf Proof.]
Note that
\[
\mat{cc}{I&0\\Y&I}\mat{cc}{V&0\\0&I}
=\mat{cc}{V&0\\0&I}\mat{cc}{I&0\\YV&I}
=\mat{cc}{V&0\\0&I}\mat{cc}{I&0\\-YV&I}^{-1},
\]
and similarly $\sbm{I&0\\X&I}\sbm{U&0\\0&I}=\sbm{U&0\\0&I}\sbm{I&0\\XU&I}^{-1}$.
It follows that $\wtilE$ and $\wtilF$ are invertible and
\begin{align*}
\wtilE\mat{cc}{V&0\\0&I}\wtilF
&=\mat{cc}{I&0\\X&I}E\mat{cc}{I&0\\Y&I}\mat{cc}{V&0\\0&I}
\mat{cc}{I&0\\-YV&I}F\mat{cc}{I&0\\-XU&I}\\
&=\mat{cc}{I&0\\X&I}E\mat{cc}{V&0\\0&I}F\mat{cc}{I&0\\-XU&I}\\
&=\mat{cc}{I&0\\X&I}\mat{cc}{U&0\\0&I}\mat{cc}{I&0\\-XU&I}
=\mat{cc}{U&0\\0&I}.
\end{align*}
Hence our first claim holds.

Now assume $E$ and $F$ are as in Corollary \ref{C:EAEspecialform}. To see that $\wtilE$ an $\wtilF$ are of the same form, it suffices to show that the operators in the right upper corner of $\wtilF$ and $\wtilF^{-1}$ coincide with those in the right upper corner of $F$ and $F^{-1}$, respectively. However, this follows directly from the formula for $\wtilF$ and the fact that
$\wtilF=\sbm{I&0\\XU&I}F^{-1}\sbm{I&0\\YV&I}$.
\end{proof}

Given operators $U$ and $V$ which are equivalent after extension via operators $E$ and $F$ as in (EAE), Lemmas \ref{L:toFS} and \ref{L:EAEtrans} both give a way to modify $E$ and $F$ to new invertible operators that still prove the equivalence after extension of $U$ and $V$. Observe that the method of Lemma \ref{L:toFS} constructs a new pair $E$ and $F$ which are of the special form of Corollary \ref{C:EAEspecialform}, while the method of Lemma \ref{L:EAEtrans} preserves the special form of Corollary \ref{C:EAEspecialform}. The combination of these two methods provides us with sufficient techniques to prove Theorem \ref{T:InvApproxCase}.

\begin{proof}[\bf Proof of Theorem \ref{T:InvApproxCase}.]
Since equivalence after extension is a symmetric relation we may assume without loss of generality that $V$ can be approximated in norm by invertible operators on $\cY$. Without loss of generality we also may assume the equivalence after extension goes via operators $E$ and $F$ as in Corollary \ref{C:EAEspecialform}, in particular, $E_{12}^{(-1)}=V$ and $F_{12}=I$. Under the transformations of Lemma \ref{L:EAEtrans} the special form of Corollary \ref{C:EAEspecialform} is preserved. In particular, we have $\wtilF_{12}=I$ invertible, so we are done if we can find operators $X$ and $Y$ such that the left lower entry $\wtilE_{21}$ of $\wtilE$ in Lemma \ref{L:EAEtrans} becomes invertible. Note that the invertibility of $E$ gives
\[
E_{21}V+E_{22}E_{22}^{(-1)}=E_{21}E_{12}^{(-1)}+E_{22}E_{22}^{(-1)}=I.
\]
Since $V$ can be approximated by invertible operators there exists an invertible operator $V_i$ on $\cY$ such that $N:=E_{21}V_i+E_{22}E_{22}^{(-1)}$ is invertible. Then
we can take $X=0$ and  $Y=E_{22}^{(-1)}V_i^{-1}$, and with this choice for $X$ and $Y$ we get $\wtilE_{21}=NV_i^{-1}$. Hence $\wtilE_{21}$ is invertible. Since $\wtilF_{12}=I$ is also
invertible, we see that $U$ and $V$ are strongly equivalent after extension, and thus, by Theorem \ref{T:SMCnSEAEeqsSC}, $U$ and $V$ are Schur coupled.
\end{proof}

The following proposition suggests that approximation in norm by invertible operators is a natural assumption in connection with equivalence after extension, in the sense that if two operators are equivalent after extension and one can be approximated in norm by invertible operators, then so can the other.

\begin{proposition}
Assume the Banach space operators $U$ on $\cX$ and $V$ on $\cY$ are equivalent after extension. Then $U$ can be approximated in norm by invertible operators on $\cX$ if and only if $V$ can be approximated in norm by invertible operators on $\cY$.
\end{proposition}

\begin{proof}[\bf Proof.]
Since equivalence after extension is a symmetric relation it suffices to prove one direction. Assume $U$ can be approximated in norm with invertible operators. Let $(U_n)_{n=0}^\infty$ be a sequence of invertible operators on $\cX$ such that $\|U-U_n\|\to 0$. Since $U$ and $V$ are equivalent after extension, we may assume they are equivalent after extension via operators $E$ and $F$ of the special form of Corollary \ref{C:EAEspecialform}, i.e., decomposed as in \eqref{EFEFinvdec} with $F_{12}=I_{\cY}$, $F_{12}^{(-1)}=I_\cX$, $E_{12}=U$ and $E_{12}^{(-1)}=V$. Because $E$ is invertible, there exists a $\de>0$ such that $\|Ev\|\geq \de\|v\|$ for each $v\in \cY\dpl\cX$, i.e., $\|E^{-1}\|\leq \de^{-1}$. Now without loss of generality we can assume $\|U-U_n\|\leq \half\de$ for each $n$. Set
\[
E_n=\mat{cc}{E_{11}&U_n\\ E_{21}& E_{22}}=E-T_n,\quad \mbox{where}\quad
T_n=\mat{cc}{0& U-U_n\\ 0&0}.
\]
As $\|T_n\|=\|U-U_n\|\leq \half\de$ for each $n$, it follows
\[
\|E_n v\|=\|E v- T_n v\|\geq \|Ev\|-\|T_nv\|\geq \half \de\|v\|.
\]
Hence $E_n$ is invertible and $\|E_n^{-1}\|\leq \frac{2}{\de}$.

Let $V_n$ be the right upper corner in the decomposition of $E_n^{-1}$, decomposed as in \eqref{EFEFinvdec}. By the standard Schur complement inversion formulas, cf., \cite{BGKR08}, we obtain that $V_n=E_{21}- E_{11}U_n^{-1}E_{22}$ and that $V_n$ is invertible, since $E_n$ is invertible. Since $E_{12}^{(-1)}=V$, we have $\|V-V_n\|\leq \|E^{-1}-E_n^{-1}\|$. Note that $E^{-1}-E_n^{-1}=-E^{-1}T_n E_n^{-1}$. Hence
\[
\|V-V_n\|\leq \|E^{-1}-E_n^{-1}\|\leq \|E^{-1}\|\|T_n\|\|E_n^{-1}\|
\leq\frac{2}{\de^2}\|T_n\|\to 0.
\]
Thus $V$ is approximated, in norm, by the sequence $(V_n)_{n=0}^\infty$ of invertible operators on $\cY$.
\end{proof}

\section{Proof of Theorem \ref{T:main} and an example}\label{S:MT}
\setcounter{equation}{0}

In the previous two sections we proved that the three operator relations of Definition \ref{D:eqrels} coincide for (a) Hilbert space operators $U$ and $V$ with closed range, and (b) Banach space operators $U$ and $V$ that can be approximated in norm by invertible operators. In both cases it suffices to check that either $U$ or $V$ is in this class of operators, since the relations imply the other operator is then also in the same class.

The question now arises if there exist operators that are not of one of these two types. It turns out that this is not the case for operators acting between separable Hilbert spaces. However, in Example \ref{E:NotTwoCases} below we will give an example of an operator acting on a non-separable Hilbert space which is of neither one of these two types.

In order to prove the main result we recall the following theorem by Feldman and Kadison \cite{FK54}.

\begin{theorem}\label{T:FK54}
A Hilbert space operator $T$ on $\cH$, with $\cH$ separable, cannot be approximated by invertible operators on $\cH$ if and only if $T=XY$ with $X$ an invertible operator on $\cH$ and $Y$ a partial isometry on $\cH$ such that the kernel and co-kernel of $X$ have different dimension.
\end{theorem}

The special form $T=XY$ in Theorem \ref{T:FK54} can be rephrased in a way that is more suitable for our purpose.

\begin{lemma}\label{L:MPdec}
A Hilbert space operator $T$ on $\cH$ is of the form $T=XY$ with $X$ an invertible operator on $\cH$ and $Y$ a partial isometry on $\cH$ if and only if $T$ has closed range.
\end{lemma}

\begin{proof}[\bf Proof.]
If $T$ is of the form $T=XY$ with $X$ on $\cH$ invertible and $Y$ on $\cH$ a partial isometry, then $X$ and $Y$ both have closed range, which implies $T$ has closed range as well.

Conversely, assume $T$ has closed range. Then $T^*$ has closed range and $T$ decomposes as
\[
T=\mat{cc}{T_1&0\\0&0}
:\mat{cc}{\im T^*\\\ker T}\to\mat{c}{\im T\\\ker T^*},
\]
with $T_1$ invertible. Now let $T_1=X_1 Y_1$ be the polar decomposition
of $T_1$ with $X_1=(T_1^*T_1)^\half$ and $Y_1$ a partial isometry from $\im T^*$ to $\im T$. Since $T_1$ invertible, so is $X_1$ and $Y_1$ is unitary. Thus we have
\[
T=\mat{cc}{T_1&0\\0&0}=\mat{cc}{X_1 Y_1&0\\0&0}=
\mat{cc}{X_1&0\\0&I_{\ker T^*}}\mat{cc}{Y_1&0\\0&0},
\]
which is the product of an invertible operator and a partial isometry.
\end{proof}

\begin{proof}[\bf Proof of Theorem \ref{T:main}.]
Combining Theorem \ref{T:InvApproxCase} and Theorem \ref{T:MPcase} shows that it suffices to prove that an operator $T$ on a separable Hilbert space $\cH$ has either closed range or can be approximated by invertible operators. However, this follows directly from Theorem \ref{T:FK54} and Lemma \ref{L:MPdec}. Hence our proof is complete.
\end{proof}

Theorem \ref{T:FK54} is a specialization of a more general result from \cite{FK54} that holds for Hilbert space operators on non-separable Hilbert spaces as well. Bouldin \cite{B90} gave a more insightful criterion, introducing the notion of essential nullity. The latter criterion is beneficial in constructing an example of an operator that is of neither one of the two types treated by Theorems \ref{T:InvApproxCase} and \ref{T:GIcase}. We briefly recall a few definitions and the main result from \cite{B90}, stated here as Theorem \ref{T:B90} below.

Let $\cH$ be a not necessarily separable Hilbert space. Let $T$ on $\cH$ have polar decompositions $T=XY$ and $T=\wtilX Y$, with $X=(T^*T)^\half$ and $\wtilX=(TT^*)^\half$,
and $Y$ a partial isometry with initial space $\im T^*$ and final space $\ov{\im T}$. For each $\vep >0$ define
\begin{align*}
\cK_\vep&=\ker ((X-\vep I)\vee 0)=\{x\in\cH \colon \|Tx\|\leq \vep\|x\|\},\\
\wtil\cK_\vep&=\ker ((\wtilX-\vep I)\vee 0)=\{x\in\cH \colon \|T^*x\|\leq \vep\|x\|\}.
\end{align*}
Then define the cardinal numbers
\[
\mbox{ess null}\, T=\inf\{\textup{dim}\, \cK_\vep\colon \vep>0\},\quad
\mbox{ess def}\, T=\inf\{\textup{dim}\, \wtil\cK_\vep\colon \vep>0\}
=\mbox{ess null}\, T^*.
\]

\begin{theorem}\label{T:B90}
A Hilbert space operator $T$ on $\cH$ can be approximated in norm by invertible
operators on $\cH$ if and only if
\[
\mbox{ess null}\, T=\mbox{ess def}\, T.
\]
\end{theorem}

\begin{example}\label{E:NotTwoCases}
Let $\cU$ be any nonseparable Hilbert space. We construct an operator $T$ on the nonseparable Hilbert space $\cH=\ell^2_+(\cU)$ in the following way. Let $\{u_1,u_2,\ldots\}$ be an orthonormal sequence in $\cU$ and let $\wtil\cU$ be the closure of the span of $u_1,u_2,\ldots$ in $\cU$. Write $\wtil\cU^\perp$ for the orthogonal complement of $\wtil\cU$ in $\cU$. Now define $T_0$ on $\cU$ by
\[
T_0=\mat{cc}{I_{\wtil\cU^\perp}&0\\0&\wtilT_0}\ons\mat{l}{\wtil\cU^\perp\\\wtil\cU},
\quad \wtilT_0 u_n=\frac{1}{n} u_n,\ \ n=1,2,\ldots.
\]
Finally set
\[
T
=\mat{ccccc}{0&0&0&0&\cdots\\T_0&0&0&0&\cdots\\0&T_0&0&0&\cdots\\
0&0&T_0&0&\cdots\\\vdots&\vdots&&\ddots&\ddots} \ons \ell^2_+(\cU).
\]
Then for any $\vep>0$ we have $\cU\oplus \{0\}\oplus \{0\}\oplus \{0\}\cdots\subset \wtil\cK_\vep(T)$, and hence $\wtil\cK_\vep(T)$ is nonseparable. On the other hand, we have $\cK_\vep(T)=\cK_\vep(T_0)\oplus\cK_\vep(T_0)\oplus\cK_\vep(T_0)\oplus\cdots$,
which is separable. Hence we can conclude that for this operator $T$
\[
\mbox{ess nul}\, T\not=\mbox{ess def}\, T.
\]
Thus $T$ cannot be approximated in norm by invertible operators. However, since $T_0$ does not have closed range, the same holds true  for $T$. Hence $T$ is not generalized invertible either.
\end{example}


\end{document}